\documentclass{amsart}
\usepackage{amssymb}
\usepackage{amsmath}
\usepackage[centertags]{amsmath}
\usepackage[english,francais]{babel}

\newtheorem{theorem}{Theorem}[section]
\newtheorem{lemma}[theorem]{Lemma}
\newtheorem{proposition}[theorem]{Proposition}
\newtheorem{corollary}[theorem]{Corollary}
\theoremstyle{definition}
\newtheorem{definition}[theorem]{Definition}

\theoremstyle{remark}
\newtheorem{remark}[theorem]{Remark}

\numberwithin{equation}{section}

\begin{document}

\title[ Stability under deformations of extremal almost-K\"ahler metrics  ]{\bf  Stability under deformations of extremal almost-K\"ahler metrics in dimension $4.$}
\author{Mehdi Lejmi}

\address{ D{\'e}partement de Math{\'e}matiques\\
UQAM\\ C.P. 8888 \\ Succ. Centre-ville \\ Montr{\'e}al (Qu{\'e}bec) \\
H3C 3P8 \\ Canada} \email{lejmi.mehdi@courrier.uqam.ca}

\maketitle
\medskip

\selectlanguage{english}
\begin{abstract}
Given a path of almost-K\"ahler metrics compatible with a fixed symplectic form on a compact $4$-manifold such that at time zero the almost-K\"ahler metric is an extremal K\"ahler one, we prove, for a short time and under a certain hypothesis, the existence of a smooth family of extremal almost-K\"ahler metrics compatible with the same symplectic form, such that at each time the induced almost-complex structure is diffeomorphic to the one induced by the path.
\end{abstract}

\section{Introduction}

An almost-K\"ahler metric on a $2n$-dimensional symplectic manifold $(M,\omega)$ is induced by an almost-complex structure $J$ compatible with $\omega$ in the sense that the tensor field $g(\cdot,\cdot)=\omega(\cdot,J\cdot)$ is symmetric and positive definite and thus it defines a Riemannian metric on $M$. The almost-K\"ahler metric is K\"ahler if the almost-complex structure $J$ is integrable. Given an almost-K\"ahler metric, one can define a canonical hermitian connection $\nabla$ (see e.g. \cite{gau1,lib}). The hermitian scalar curvature $s^\nabla$ is then obtained by taking a trace and contracting the curvature of $\nabla$ with $\omega$. In the K\"ahler case, the hermitian scalar curvature coincides with the Riemannian scalar curvature.

A key observation, made by Fujiki \cite {fuj} in the integrable case and by Donaldson \cite{don} in the general almost-K\"{a}hler case, asserts that the natural action of the infinite dimensional group $Ham(M,\omega)$ of hamiltonian symplectomorphisms on the space $AK_{\omega}$ of $\omega$-compatible almost-K\"ahler metrics is hamiltonian with moment map $\mu : AK_{\omega}\rightarrow \left(Lie(Ham(M,\omega))\right)^\ast$ given by $\mu_J(f)=\int_M s^{\nabla}f\frac{\omega^n}{n!}$. The critical points of the norm $\int_M {\left(s^{\nabla}\right)}^2\frac{\omega^n}{n!}$ are called {\it extremal almost-K\"ahler metrics}. It turns out that the symplectic gradient of $s^\nabla$ of such metrics is a holomorphic vector field in the sense that its flow preserves the corresponding almost-complex structure. In particular, extremal K\"ahler
metrics in the sense of Calabi \cite{cal} and almost-K\"ahler metrics with constant hermitian scalar curvature are extremal.
 
The GIT formal picture in \cite{don} suggests the existence and the uniqueness of an extremal almost-K\"ahler metric, modulo the action of $Ham(M,\omega)$, in each `stable complexified' orbit of the action of $Ham(M,\omega)$.
However, in this formal infinite dimensional setting, a natural complexification of $Ham(M,\omega)$ does not exist. When $H^1(M,\mathbb{R})=0$, an identification of the `complexified' orbit of a K\"ahler metric $(J,g)\in AK_\omega$ is given by considering all K\"ahler metrics $(J,\tilde{g})$ in the K\"ahler class $\left[\omega\right]$ and applying Moser's Lemma \cite{don}. In this setting, Fujiki--Schumacher \cite{fuj-sch} and LeBrun--Simanca \cite{leb-sim} showed, in the abscence of holomorphic vector fileds, that the existence of an extremal K\"ahler metric is an open condition on the space of such orbits. Moreover,
Apostolov--Calderbank--Gauduchon--Friedman \cite{apo-cal-gau-Fri 2} generalized this result by fixing a maximal torus $T$ in the reduced automorphism group of $(M,J)$ and considering $T$-invariant $\omega$-compatible K\"ahler metrics. In general, for an almost-K\"ahler metric, a description of these `complexified' orbits is not avaible, see however \cite{don1} for the toric case.
Nevertheless,  the formal picture suggests that the existence of an extremal K\"ahler metric should persist for smooth almost-K\"ahler metrics close to an extremal one.

Thus motivated, we consider in this paper the $4$-dimensional case where one can introduce a notion of almost-K\"ahler potential related to the one defined by Weinkove \cite{tos-wei-yau,wei}. In the spirit of \cite{fuj-sch,leb-sim}, we shall apply the Banach Implicit Function Theorem for the hermitian scalar curvature of $T$-invariant $\omega$-compatible almost-K\"ahler metrics where $T$ is a maximal torus in $Ham(M,\omega).$
The main technical problem is the regularity of a family of Green operators involved in the definition of the almost-K\"ahler potential. Using a Kodaira--Spencer result \cite{kod,kod-mor}, one can resolve this problem if we suppose that
the dimension of $g_t$-harmonic $J_t$-anti-invariant $2$-forms, denoted by $h^{-}_{J_t}$ (see \cite{dra-li-zha}), satisfies the condition $h^{-}_{J_t}=h^{-}_{J_0}=b^+(M)-1$ for $t\in\left(-\epsilon,\epsilon\right)$ along the path $(J_t,g_t)\in AK^T_\omega$ in the space of $T$-invariant $\omega$-compatible almost-K\"ahler metrics. So, our main theorem claims the following

\begin{theorem}\label{th1}
Let $(M,\omega)$ be a $4$-dimensional compact symplectic manifold and $T$ a maximal torus in $Ham(M,\omega)$.
Let $(J_t,g_t)$ be any smooth family of almost-K\"ahler metrics in $AK^T_\omega$ such that $(J_0,g_0)$ is an extremal {{K\"ahler}} metric. Suppose that $h^{-}_{J_t}=h^{-}_{J_0}=b^+(M)-1$ for $t\in\left(-\epsilon,\epsilon\right).$ Then, there exists a smooth family $({\tilde{J}}_t,{\tilde{g}}_t)$ of extremal almost-K\"ahler metrics
in $AK_\omega^T$, defined for sufficiently small $t$, with $({\tilde{J}}_0,{\tilde{g}}_0)=(J_0,g_0)$ and such that ${\tilde{J}}_t$ is equivariantly diffeomorphic to $J_t.$ 
\end{theorem}
\begin{remark}
(i) The condition that $h^{-}_{J_t}=h^{-}_{J_0}=b^+(M)-1$ for $t\in\left(-\epsilon,\epsilon\right)$ is satisified in the following cases:

\begin{enumerate}
\item When $J_t$ are integrable almost-complex structures for each $t$. Then, $h^{-}_{J_t}=2h^{2,0}(M,J_t)=b^+(M)-1$ by a well-known result of Kodaira \cite{bar-pet-van}. On the other hand, it is unknown
whether or not, for an $\omega$-compatible {\it non-integrable} almost-complex $J$ on a compact $4$-dimensional symplectic manifold $M$ with $b^+(M)\ge 3$, the equality $h^{-}_{J}=b^+(M)-1$ is possible (see \cite{dra-li-zha}).
\item When $b^+(M)=1$, $h^{-}_{J_t}=0$ for each $t$. This condition is satisfied when $(M,\omega)$ admits a non trivial torus in $Ham(M,\omega)$ \cite{kar}.

\end{enumerate}

\noindent (ii) Theorem \ref{th1} holds under the weaker assumption that the torus $T\subset Ham(M,\omega)$ is maximal in $Ham(M,\omega)\cap Isom_0(M,g_0)$, where $ Isom_0(M,g_0)$ denotes the connected component of the isometry group of the initial metric $g_0.$ By a known result of Calabi \cite{cal2}, any extremal K\"ahler metric is invariant under a maximal connected compact subgroup of $Ham(M,\omega)\cap \widetilde{Aut}(M,J_0)$, where $ \widetilde{Aut}(M,J_0)$ is the reduced automorphism group of $(M,J_0)$. Hence, Theorem \ref{th1} generalizes \cite{fuj-sch,leb-sim} in the $4$-dimensional case.

\noindent (iii) It was kindly pointed out to us by T. Dr\u{a}ghici that using a recent result of Donaldson and Remarks (i) and (ii) above, one can further extend Theorem \ref{th1} in the case when $b^+(M)=1$ as follows: Let
$(M,\omega_0,J_0,g_0)$ be a compact $4$-dimensional extremal K\"ahler manifold with $b^+(M)=1$ and $T$ be a maximal torus in ${Ham}(M,\omega)\cap {{Isom}_0}(M,g_0)$.
Then, for any smooth family of $T$-invariant almost-complex structures $J(t)$ with $J(t)=J_0$, $J(t)$ is compatible with an extremal almost-K\"ahler metric $g_t$ for $t\in\left(-\epsilon,\epsilon\right)$. Indeed, as $J(t)$ are tamed by $\omega_0$ for $t\in\left(-\epsilon,\epsilon\right)$
and $b^+(M)=1$, one can use the openess result of Donaldson \cite[Proposition 1]{don2} (see also \cite[Sec. 5]{dra-li-zha}) to show that there exists a smooth family of $J(t)$-invariant symplectic forms $\omega_t$ with $\left[\omega_t\right]=\left[\omega_0\right]$.
Averaging $\omega_t$ over the compact group $T$ and using the equivariant Moser Lemma, we obtain a family ${{J}}_t$ of $T$-invariant $\omega_0$-compatible almost-complex structures such that ${{J}}_t$ is $T$-equivariantly diffeomorphic to $J(t)$.
We can then apply Theorem \ref{th1} to produce compatible extremal metrics
\end{remark}

Kim and Sung \cite{kim-sun} showed that, in any dimension, if one starts with a K\"ahler metric of constant scalar curvature with no holomorphic vector fields, one can construct infinite dimensional families of
almost-K\"ahler metrics of constant hermitian scalar curvature which concide with the initial metric away from an open set. Similar existence result was presented in \cite{lej} when the initial K\"ahler metric is locally toric.
\section{Preliminaries}

Let $(M,\omega)$ be a compact symplectic manifold of dimension $2n$. An almost-complex structure $J$ is compatible with $\omega$ if the tensor field $g(\cdot,\cdot):= \omega(\cdot,J\cdot)$ defines a Riemannian metric on
$M$; then, $(J,g)$ is called an ($\omega$-compatible) {\it almost-K\"ahler} metric on $(M,\omega)$. If, additionally, the almost-complex structure $J$ is {\it{integrable}}, then $(J,g)$ is a {\it K\"ahler} metric on $(M,\omega).$

The almost-complex structure $J$ acts on the cotangent bundle $T^\ast(M)$ by $J\alpha(X)=-\alpha(JX),$ where $\alpha$ is a $1$-form and $X$ a vector field on $M$. Any section $\psi$ of the bundle $\otimes^2T^\ast(M)$ admits an orthogonal splitting $\psi=\psi^{J,+}+\psi^{J,-}$, where $\psi^{J,+}$ is the $J$-invariant part and $\psi^{J,-}$ is the $J$-anti-invariant part,
given by
\begin{equation*}
\psi^{J,+}(\cdot,\cdot)=\frac{1}{2}\left(\psi(\cdot,\cdot)+\psi(J\cdot,J\cdot)\right){\text{ and }} \psi^{J,-}(\cdot,\cdot)=\frac{1}{2}\left(\psi(\cdot,\cdot)-\psi(J\cdot,J\cdot)\right).
\end{equation*}
In particular, the bundle of $2$-forms decomposes under the action of $J$
\begin{equation}\label{split1}
\Lambda^2(M)=\mathbb{R}\,.\,\omega\oplus\Lambda_0^{J,+}(M)\oplus\Lambda^{J,-}(M),
\end{equation}
where $\Lambda_0^{J,+}(M)$ is the subbundle of the {\it primitive} $J$-invariant $2$-forms (i.e. $2$-forms pointwise orthogonal to $\omega$) and $\Lambda^{J,-}(M)$ is the subbundle of $J$-anti-invariant $2$-forms. Hence, the subbundle of primitive $2$-forms $\Lambda_0^2(M)$ admits the splitting
\begin{equation*}
\Lambda^2_0(M)=\Lambda_0^{J,+}(M)\oplus\Lambda^{J,-}(M).
\end{equation*}

For an $\omega$-compatible almost-K\"ahler metric $(J,g)$, the canonical hermitian connection on the complex tangent bundle $(T(M),J,g)$ is defined by
\begin{equation*}
\nabla_XY=D^g_XY-\frac{1}{2}J\left(D^g_XJ\right)Y,
\end{equation*}
where $D^g$ is the Levi-Civita connection with respect to $g$ and $X,Y$ are vector fields on $M$. Denote by $R^{\nabla}$ the curvature of $\nabla$. Then, the {\it hermitian Ricci form} $\rho^{\nabla}$ is the trace of $R^{\nabla}_{X,Y}$ viewed as an anti-hermitian linear operator of $(T(M),J,g)$, i.e. 
\begin{equation*}
\rho^{\nabla}(X,Y)=-tr(J\circ R^{\nabla}_{X,Y}). 
\end{equation*}
Hence, the $2$-form $\rho^{\nabla}$ is a closed (real) $2$-form and it is a deRham representative of $2\pi c_1(T(M),J)$ in $H^2(M,\mathbb{R})$, where $ c_1(T(M),J)$ is the first (real) Chern class. If the almost-complex structure $J$ is compatible with a symplectic form $\tilde{\omega}$ such that ${\tilde{\omega}}^n=e^{F}{\omega^n}$ for some smooth real-valued function $F$ on $M$, then \cite{tos-wei,tos-wei-yau}
\begin{equation}\label{equa5}
{\tilde{\rho}}^\nabla=-\frac{1}{2}dJdF+{{\rho}}^\nabla,
\end{equation}
where ${\tilde{\rho}}^\nabla$ is the hermitian Ricci form of the almost-K\"ahler metric $(J,\tilde{g})$ (here $\tilde{g}(\cdot,\cdot)=\tilde{\omega}(\cdot,J\cdot)$ is the induced Riemannian metric).

We define the {\it hermitian scalar curvature} $s^{\nabla}$ of an almost-K\"{a}hler metric $(J,g)$ as  the trace of $\rho^\nabla$ with respect to $\omega$, i.e.
\begin{equation}\label{equa6}
s^{\nabla}\omega^n=2n\left(\rho^{\nabla}\wedge\omega^{n-1}\right).
\end{equation}
%

The (Riemannian) {\it Hodge operator} $\ast_g:\Lambda^p(M)\rightarrow\Lambda^{2n-p}(M)$ is defined to be the unique isomorphism such that $\psi_1\wedge\left(\ast_g\,\psi_2\right)=g(\psi_1,\psi_2)\,\frac{\omega^n}{n!}$, for any $p$-forms $\psi_1,\psi_2.$
Then, the codifferential $\delta^g$, defined as the formal adjoint of the exterior derivative $d$ with respect to $g$, is related to $d$ by the relation \cite{bes,gau} 
\begin{equation*}
\delta^g=-\ast_{g}d\,\ast_g.
\end{equation*}
 It follows that
\begin{equation}\label{equa3}
d=\ast_g\,\delta^g\,\ast_g.
\end{equation}

In dimension $2n=4$, the bundle of $2$-forms decomposes as
\begin{equation*}
\Lambda^2(M)=\Lambda^+(M)\oplus\Lambda^-(M),
\end{equation*}
where $\Lambda^{\pm}(M)$ correspond to the eigenvalue $(\pm1)$ under the action of the Hodge operator $\ast_g$. This decomposition is related to the splitting (\ref{split1}) as follows
\begin{equation}\label{split2}
\Lambda^+(M)=\mathbb{R}\,.\,\omega\oplus\Lambda^{J,-}(M) {\text{ and }}\Lambda^-(M)=\Lambda_0^{J,+}(M).
\end{equation}

\section{Extremal almost-K\"ahler metrics}\label{sec}
Let $(M,\omega)$ be a compact and connected symplectic manifold of dimension $2n$. Any $\omega$-compatible almost-complex structure is identified with the induced Riemannian metric.

Denote by $AK_{\omega}$ the Fr\'echet space of $\omega$-compatible almost-complex structures.
The space $AK_{\omega}$ comes naturally equipped with a formal K\"{a}hler structure. Let $Ham(M,\omega)$ be the group of hamiltonian symplectomorphisms of $(M^{2n},\omega)$. The Lie algebra of $Ham(M,\omega)$ is identified with the space of smooth functions on $M$ with zero mean value.

A key observation, made by Fujiki \cite {fuj} in the integrable case and by Donaldson \cite{don} in the general almost-K\"{a}hler case, asserts that the natural action of $Ham(M,\omega)$ on $AK_{\omega}$ is hamiltonian with momentum given by the hermitian scalar curvature. More precisely, the moment map $\mu : AK_{\omega}\rightarrow \left(Lie(Ham(M,\omega))\right)^\ast$ is
\begin{equation*}
\mu_J(f)=\int_M s^{\nabla}f\frac{\omega^n}{n!}
\end{equation*}
where $s^{\nabla}$ is the hermitian scalar curvature of $(J,g)$ and $f$ is a smooth function with zero mean value viewed as an element of $Lie(Ham(M,\omega))$. The square-norm of the hermitian
scalar curvature defines a functional on $AK_{\omega}$ 
\begin{equation}\label{equa4}
J\mapsto\int_M {\left(s^{\nabla}\right)}^2\frac{\omega^n}{n!}.
\end{equation}
\begin{definition}
The critical points $(J,g)$ of the functional (\ref{equa4}) are called {\it extremal almost-K\"{a}hler metrics}. 
\end{definition}
\begin{proposition}\label{prop7}
An almost-K\"ahler metric $(J,g)$ is a critical point of (\ref{equa4}) if and only if $grad_\omega s^{\nabla}$ is a Killing vector field with respect to $g$.
\end{proposition}
A proof of Proposition \ref{prop7} is given in \cite{apo-dra,gau,lej}.

\subsection{The extremal vector field}

We fix a maximal torus $T$ in $Ham(M,\omega)$ and denote by $\mathfrak{t}_{\omega}$ the finite dimensional
space of real-valued smooth functions on $M$ which are hamiltonians with zero mean value of elements of
$\mathfrak{t}=Lie(T).$ 
Denote by $\Pi_{\omega}^T$ the $L^2$-orthogonal
projection of $T$-invariant smooth functions on $\mathfrak{t}_{\omega}$ with
respect to the volume form $\frac{\omega^n}{n!}$. Let $AK_{\omega}^T$ be the space of $\omega$-compatible $T$-invariant almost-complex structures.
Given any $J\in AK_{\omega}^T$, we define $z^T_{\omega}:=\Pi_{\omega}^T s^{\nabla},$ where $s^{\nabla}$ is the hermitian scalar curvature of $(J,g)$. Then, we have the following (for more details see \cite{apo-cal-gau-Fri 2,gau,lej})

\begin{proposition}\label{prop1}
The potential $z^T_{\omega}$ is independant of $(J,g).$ Furthermore, a $\omega$-compatible $T$-invariant almost-K\"ahler metric $(J,g)$ is extremal if and only if
\begin{equation*}
\mathring{s}^\nabla=z^T_{\omega},
\end{equation*}
where $\mathring{s}^\nabla$ is the integral zero part of the hermitian scalar curvature $s^\nabla$ of $(J,g).$
\end{proposition}
\begin{definition}\label{def1}
The vector field $Z^T_{\omega}:=grad_{\omega}z^T_{\omega}$ is called the {\it extremal vector field} relative to $T$. 
\end{definition}
\begin{proposition}\label{prop2}
The vector field $Z^T_{\omega}$ is invariant under $T$-invariant isotopy of $\omega.$
\end{proposition}
\begin{remark}
The assumption that $T\subset Ham(M,\omega)$ is a maximal torus is used only in the second part of Proposition \ref{prop1}. Indeed, the arguments in \cite{lej} show that $z^T_{\omega}=\Pi_{\omega}^T s^{\nabla}$ is independent of $(J,g)$
for any torus $T\subset Ham(M,\omega)$ and Proposition \ref{prop2} still holds true for the corresponding vector field $Z^T_{\omega}=grad_{\omega}z^T_{\omega}$.

\end{remark}

\section{Almost-K\"ahler potentials in dimension 4}

Let $(M,\omega)$ be a compact symplectic manifold of dimension $2n=4$ and $(J,g)$ a $\omega$-compatible almost-K\"ahler metric. In order to define the almost-K\"ahler potentials, we consider the following second order linear differential operator \cite{lej1} on the smooth sections $\Omega^{J,-}(M)$ of the bundle of $J$-anti-invariant $2$-forms.

\begin{equation*}
\begin{array}{lrcl}
P : & \Omega^{J,-}(M) & \longrightarrow & \Omega^{J,-}(M) \\
    & \psi & \longmapsto & (d\delta^g\psi)^{J,-},
\end{array}
\end{equation*}
where $\delta^g$ is the codifferential with respect to the metric $g.$
\begin{lemma}
$P$ is a self-adjoint strongly elliptic linear operator with kernel the $g$-harmonic $J$-anti-invariant $2$-forms.
\end{lemma}
\begin{proof}
The principal symbol of $P$ is given by the linear map $\sigma(P)_\xi(\psi)=-\frac{1}{2}|\xi|^2\psi,$ $\forall\xi\in T^\ast_x(M),\psi\in\Omega^{J,-}(M)$. So, $P$ is a self-adjoint elliptic linear operator with respect to the global inner product $\left<\cdot,\cdot\right>=\int_Mg(\cdot,\cdot)\,\frac{\omega^2}{2}.$ Now, let $\psi\in \Omega^{J,-}(M)  $ and suppose that $P(\psi)=0$. Then, $0=\left< (d\delta^g\psi)^{J,-},\psi\right>=\left< d\delta^g\psi,\psi\right>=\left<\delta^g\psi,\delta^g\psi\right>$ which means that $\delta^g\psi=0.$ It follows from (\ref{split2}) and since $\psi$ is $J$-anti-invariant that $\ast_g\psi=\psi$. Using the relation (\ref{equa3}), we obtain $d\psi=\ast_g\delta^g\ast_g\psi=\ast_g\delta^g\psi=0$. Hence, $d\psi=\delta^g\psi=0$ and thus $\psi$ is a $g$-harmonic $J$-anti-invariant $2$-form.
\end{proof}
\begin{corollary}\label{cor1}
For $f\in C^\infty(M,\mathbb{R})$, there exist a unique $\psi_f\in \Omega^{J,-}(M)$ orthogonal to the kernel of $P$ such that $(d\delta^g\psi_f)^{J,-}=(dJdf)^{J,-}$.
\end{corollary}
\begin{proof}
For a smooth real-valued function $f\in C^\infty(M,\mathbb{R})$ and any $\alpha$ in the kernel of $P$, we have $\left<(dJdf)^{J,-},\alpha\right>=\left<dJdf,\alpha\right>=\left<Jdf,\delta^g\alpha\right>=0.$ By a standard result of elliptic theory \cite{bes, wel} and since $P$ is self-adjoint, there exist a smooth section $\psi_f\in\Omega^{J,-}(M)$ such that $P(\psi_f)=(dJdf)^{J,-}$. Moreover, $\psi_f$ is unique if one requires $\psi_f$ be orthogonal to the kernel of $P$.

\end{proof}

From Corollary \ref{cor1}, it follows that, for $f\in C^\infty(M,\mathbb{R})$, the symplectic form $\omega_f=\omega+d(Jdf-\delta^{g}\psi_f)$ is a $J$-invariant closed $2$-form. Then, the function $f$ is called an {\it almost-K\"ahler potential} if the induced symmetric
tensor $g_f(\cdot,\cdot):=\omega_f(\cdot,J\cdot)$ is a Riemannian metric. This notion of almost-K\"ahler potential is closely related but different (in general) from the one defined by Weinkove in \cite{wei}. More precisely, if the almost-complex structure $J$ is compatible with a symplectic form $\tilde{\omega}$ which is cohomologous to $\omega$ i.e. $\tilde{\omega}-\omega=d{\alpha}$ (for some $1$-form ${\alpha}$), then the almost-K\"ahler potential defined by Weinkove is given by the function $\tilde{f}$ which is uniquely determined (up to the addition of constant) by the Hodge decomposition of ${\alpha}$ with respect to the (self-adjoint elliptic) {\it twisted Laplace operator} $\tilde{\Delta}^{{c}}=J\Delta^{\tilde{g}}J^{-1},$ where $\Delta^{\tilde{g}}$ is the (Riemannian) Laplace operator with respect to the induced metric $\tilde{g}(\cdot,\cdot)=\tilde{\omega}(\cdot,J\cdot).$ In other words, we have the decomposition ${\alpha}={\alpha}_{H^c}+\tilde{\Delta}^{{c}}\tilde{\mathbb{G}}\,{\alpha}$, where $\tilde{\mathbb{G}}$ is the {\it Green operator} associated to $\tilde{\Delta}^{{c}}$ and ${\alpha}_{H^c}$ is the harmonic part of ${\alpha}$ with respect to $\tilde{\Delta}^{{c}}$. Thus, $\tilde{f}=-\delta^{\tilde{g}}J\,\tilde{\mathbb{G}}\alpha,$ where $\delta^{\tilde{g}}$ is the codifferential with respect to the metric $\tilde{g}.$

Note that $(dJdf)^{J,-}=D^g_{{(df)}^{\sharp_g}}\omega$ (see e.g. \cite{gau}), where $\sharp_g$ stands for the isomorphism
between $T^\ast(M)$ and $T(M)$ induced by $g^{-1}$. Hence, in the K\"ahler case, $(dJdf)^{J,-}=0$ which implies that $\psi_f=0$ and thus this almost-K\"ahler potential coincides with the usual K\"ahler one.

\section{Main Theorem}
Let $(M,\omega)$ be a compact and connected symplectic manifold of dimension $2n=4$ and $J_t\in AK_\omega$ be a smooth path of $\omega$-compatible almost-complex structures. We define the following family of differential operators associated to $J_t$
\begin{equation*}
\begin{array}{lrcl}
P_t : & \Omega_0^2(M) & \longrightarrow & \Omega_0^2(M) \\
    & \psi & \longmapsto & \frac{1}{2}\Delta^{g_t} \psi-\frac{1}{4}g_t(\Delta^{g_t}\psi,\omega)\omega,
\end{array}
\end{equation*}
where $\Omega_0^2(M)$ is the space of smooth sections of the bundle $\Lambda_0^2(M)$ of primitive $2$-forms (pointwise orthogonal to $\omega$) and $\Delta^{g_t}$ is the (Riemannian) Laplacian with respect to the metric $g_t(\cdot,\cdot)=\omega(\cdot,J_t\cdot)$ (here we use the convention $g_t(\omega,\omega)=2$).

One can easily check that $P_t$ preserves the decomposition 
\begin{equation*}
\Omega^2_0(M)=\Omega_0^{J,+}(M)\oplus\Omega^{J,-}(M).
\end{equation*}
Furthermore,
\begin{equation*}
P_t|_{ \Omega^{J_t,-}(M)}(\psi)=(d\delta^{g_t}\psi)^{J_t,-} {\text{  and  }}P_t|_{ \Omega_0^{J_t,+}(M)}(\psi)= \frac{1}{2}\Delta^{g_t} \psi.
\end{equation*}
It follows that the kernel of $P_t$ consists of primitive harmonic $2$-forms which splits as {\it anti-selfdual}
and $J_t$-anti-invariant ones so we have
\begin{equation*}
{\text{dim ker}}(P_t)=b^-(M)+h^{-}_{J_t},
\end{equation*}
where $h^{-}_{J_t}$ is introduced by Dr\u{a}ghici--Li--Zhang in \cite{dra-li-zha}. 

Moreover, $P_t- \frac{1}{2}\Delta^{g_t}$ is a linear differential operator of order $1$. Indeed, a direct computation shows that
\begin{eqnarray*}
\left(P_t- \frac{1}{2}\Delta^{g_t}\right)(\psi)&=&\frac{1}{2}\left[\frac{1}{2}\delta^{g_t}\left(D^{g_t}\omega(\psi)\right)-\frac{1}{2}g_t(D^{g_t}\psi,D^{g_t}\omega)\right.\\
&+&\left.\frac{s_{g_t}}{6}g_t(\omega,\psi)-W^{g_t}(\omega,\psi)\right]\omega,
\end{eqnarray*}
where $W^{g_t}$ stands for the {\it Weyl tensor} (see e.g. \cite{bes}), $D^{g_t}$ (resp. $\delta^{g_t}$) for the Levi-Civita connection (resp. the codifferential) with respect to the metric $g_t$ and $s_{g_t}$ for the Riemannian scalar curvature defined as the trace of the (Riemannian) tensor.

The operator $P_t$ is a self-adjoint strongly elliptic linear operator of order 2. We obtain then a family of Green operators ${\mathbb{G}}_t$ associated to $P_t$. 
If $h^{-}_{J_t}=h^{-}_{J_0}=b^+(M)-1$ for $t\in\left(-\epsilon,\epsilon\right)$, then ${\mathbb{G}}_t $ is $C^\infty$ {\it differentiable} in $t\in\left(-\epsilon,\epsilon\right)$ \cite{kod,kod-mor}, meaning that ${\mathbb{G}}_t(\psi_t)$ is a smooth family of sections of $\Lambda^2_0(M)$
for any smooth sections $\psi_t.$


To show Theorem \ref{th1}, we consider the extension of ${\mathbb{G}}_t $ to the Sobolev spaces $W^{k,p}(M, \Lambda_0^2(M))$ involving derivatives up to $k$.
\begin{lemma}\label{lem5}
Let ${\mathbb{G}}_t : \Omega_0^2(M) \rightarrow\Omega_0^2(M)$ the family of the above Green operators associated to $P_t$ and suppose that $h^{-}_{J_t}=h^{-}_{J_0}=b^+(M)-1$ for $t\in\left(-\epsilon,\epsilon\right)$. Then, the extension of ${\mathbb{G}}_t$ to Sobolev spaces, still denoted by ${\mathbb{G}}_t$, defines a $C^1$ map ${\mathbb{G}}:  (-\epsilon,\epsilon)\times  {W}^{p,k}(M, \Lambda_0^2(M))\rightarrow {W}^{p,k+2}(M, \Lambda_0^2(M))$ 
\end{lemma}
\begin{proof}
Denote by $\Pi_t$ the $L^2$-orthogonal projection to the kernel of $P_t$ with respect to $\left<\cdot,\cdot\right>_{L^2_{g_t}}=\int_Mg_t(\cdot,\cdot)\frac{\omega^2}{2}$. 
We claim that ${\mathbb{G}}_t\circ \Pi_0$ and $\Pi_0 \circ {\mathbb{G}}: (-\epsilon, \epsilon) \times W^{k,p}(M, \Lambda_0^2(M)) \rightarrow W^{k+2,p}(M,\Lambda_0^2(M))$ are $C^1$ maps. Indeed, let $\{\psi^i_{0}\}$ be an orthonormal basis of the kernel of $P_0$
with respect to $\left<\cdot,\cdot\right>_{L^2_{g_0}}$.
Note that $\psi^i_{0}$ are smooth since $P_0$ is elliptic. Then, we have 
\begin{eqnarray*}
\left({\mathbb{G}}_t \circ \Pi_0\right) (\psi) &=& \sum_i \left<\psi, \psi^i_{0}\right>_{L^2_{g_0}} {\mathbb{G}}_t(\psi^i_{0}),\\
\left(\Pi_0 \circ {\mathbb{G}}_t\right)(\psi)&=& \sum_i  \left<{\mathbb{G}}_t(\psi), (\psi^i_{0})^{J_0,+}+(\psi^i_{0})^{J_0,-}\right>_{L^2_{g_0}}\psi^i_{0}\\
&=& \sum_i \left( \int_M-{{\mathbb{G}}_t(\psi)\wedge(\psi^i_{0})^{J_0,+}}+{{\mathbb{G}}_t(\psi)\wedge(\psi^i_{0})^{J_0,-}}\right)\psi^i_{0}\\
&=& \sum_i  \left(\int_M-{\mathbb{G}}_t(\psi)\wedge\left((\psi^i_{0})^{J_0,+}\right)^{J_t,+}-{\mathbb{G}}_t(\psi)\wedge\left((\psi^i_{0})^{J_0,+}\right)^{J_t,-}\right.\\
&+&\left. {\mathbb{G}}_t(\psi)\wedge\left((\psi^i_{0})^{J_0,-}\right)^{J_t,+}+{\mathbb{G}}_t(\psi)\wedge\left((\psi^i_{0})^{J_0,-}\right)^{J_t,-}\right) \psi^i_{0}\\
&=& \sum_i  \left[\left<\psi,{\mathbb{G}}_t \left(\left((\psi^i_{0})^{J_0,+}\right)^{J_t,+}\right)\right>_{L^2_{g_t}} -\left<\psi,{\mathbb{G}}_t \left(\left((\psi^i_{0})^{J_0,+}\right)^{J_t,-}\right)\right>_{L^2_{g_t}}\right.\\
&-&\left. \left<\psi,{\mathbb{G}}_t \left(\left((\psi^i_{0})^{J_0,-}\right)^{J_t,+}\right)\right>_{L^2_{g_t}}+ \left<\psi,{\mathbb{G}}_t\left(\left((\psi^i_{0})^{J_0,-}\right)^{J_t,-}\right)\right>_{L^2_{g_t}}\right] \psi^i_{0}
\end{eqnarray*}
(in the latter equality, we used the fact that ${\mathbb{G}}_t$ is self-adjoint with respect to $L^2_{g_t}$). The claim follows from the result of Kodaira--Spencer \cite{kod,kod-mor}.

Denote by $W^{k,p}(M, \Lambda_0^2(M))^{\perp}$ the space of $2$-forms in $W^{k,p}(M, \Lambda_0^2(M))$ which are orthogonal to the kernel of $P_0$ with respect to ${L^2_{g_0}}$ and consider the map
\begin{equation*}
\begin{array}{lrcl}
\Phi: & (-\epsilon, \epsilon) \times W^{k+2,p}(M, \Lambda_0^2(M))^{\perp}& \longrightarrow &  (-\epsilon, \epsilon) \times W^{k,p}(M,\Lambda_0^2(M))^{\perp} \\
    & (t,\psi) & \longmapsto & (t,(Id - \Pi_0)P_t(\psi)),
\end{array}
\end{equation*}
Clearly, the map $\Phi$ is of class $C^1$ and its differential at $(0,\psi)$ is an isomorphism so by the {\it inverse function theorem} for Banach spaces there exist a neighboorhood $V$ of $(0,\psi)$ such that $\Phi|_V$ admits an inverse of class $C^1$. By the The Kodaira--Spencer result \cite{kod,kod-mor}, the map $\Pi:  (-\epsilon,\epsilon)\times W^{k,p}(M, \Lambda_0^2(M))\rightarrow W^{k,p}(M, \Lambda_0^2(M))$ is $C^1$ and thus the map $P_t (Id-\Pi_0) G_t (Id -\Pi_0) = (Id - \Pi_t)(Id - \Pi_0)-P_t (\Pi_0 G_t)(Id - \Pi_0)$ is clearly $C^1$ since it is a composition of such operators. Then, the map
\begin{eqnarray*}
\Phi|_V^{-1}\left(t, (Id-\Pi_0)P_t (Id-\Pi_0){\mathbb{G}}_t (Id -\Pi_0)\right)&=& (t,(Id-\Pi_0){\mathbb{G}}_t(Id -\Pi_0)) \\
&=&(t,  {\mathbb{G}}_t -\Pi_0{\mathbb{G}}_t-{\mathbb{G}}_t \Pi_0 + \Pi_0{\mathbb{G}}_t\Pi_0)
\end{eqnarray*}
is $C^1$ and hence ${\mathbb{G}}_t$ is $C^1$.

\end{proof}




{\it Proof of Theorem \ref{th1}} Let $(M,\omega)$ be a $4$-dimensional compact and connected symplectic manifold and $T$ a maximal torus in $Ham(M,\omega)$.
Let $(J_t,g_t)$ a smooth family of $\omega$-compatible almost-K\"ahler metrics in $AK^T_\omega$ such that $(J_0,g_0)$ is an extremal K\"ahler metric.

Following \cite{leb-sim}, we consider the almost-K\"ahler deformations
\medskip
\begin{equation*}
\omega_{t,f}=\omega+d(J_tdf-\delta^{g_t}\psi_f^t),
\end{equation*}
where $f$ belongs to the
Fr\'echet space $\widetilde{C}^\infty_T(M,\mathbb{R})$ of $T$-invariant smooth functions (with zero
integral), which are $L^2$-orthogonal, with respect to $\frac{\omega^2}{2}$,
to $\mathfrak{t}_\omega$ and where the $2$-form $\psi_f^t$ is given by Corollary \ref{cor1}.


Let $\mathcal{U}$ be an open set in
$\mathbb{R}\times\widetilde{C}^\infty_T(M,\mathbb{R})$ containing $(0,0)$ such that the symmetric tensor $g_{t,f}(\cdot,\cdot):=\omega_{t,f}(\cdot,J_t\cdot)$ is a Riemannian metric. 

By possibly replacing $\mathcal{U}$ with a smaller open set, we may assume as in \cite{leb-sim} that the kernel of the operator $(Id-\Pi_{\omega}^T)\circ(Id-\Pi_{\omega_{t,f}}^T)$ is equal to the kernel of $(Id-\Pi_{\omega_{t,f}}^T).$ Indeed, let $\{X_1,\cdots,X_n\}$ be a basis of $\mathfrak{t}=Lie(T)$.  Then, the corresponding hamiltonians with zero mean value $\{\xi_\omega^1,\cdots,\xi_\omega^n\}$ resp. $\{\xi_{\omega_{t,f}}^1,\cdots,\xi_{\omega_{t,f}}^n\}$, with respect to $\omega$ resp. $\omega_{t,f}$,  form a basis of $\mathfrak{t}_\omega$ resp. $\mathfrak{t}_{\omega_{t,f}}.$
Let $\{{\tilde{\xi}}_\omega^1,\cdots,{\tilde{\xi}}_\omega^n\}$ resp. $\{{\tilde{\xi}}_{\omega_{t,f}}^1,\cdots,{\tilde{\xi}}_{\omega_{t,f}}^n\}$ the corresponding orthonormal basis obtained by the Gram--Schmidt procedure. Since $\det\left[\left<{\tilde{\xi}}_\omega^i,{\tilde{\xi}}_{\omega_{t,f}}^j\right>\right]$ defines a continuous function on $\mathcal{U}$, then we may suppose that $\det\left[\left<{\tilde{\xi}}_\omega^i,{\tilde{\xi}}_{\omega_{t,f}}^j\right>\right]\neq 0$ on an eventually smaller open set than $\mathcal{U}$ (here $\left<\cdot,\cdot\right>$ denotes the $L^2$ product with respect to the volume form $\frac{\omega_{t,f}^2}{2}$). So, if $u\in ker\left((Id-\Pi_{\omega}^T)\circ(Id-\Pi_{\omega_{t,f}}^T)\right)$ then $v\in\mathfrak{t}_\omega\cap\left(\mathfrak{t}_{\omega_{t,f}}\right)^{\bot_{g_{t,f}}}$, where $v=(Id-\Pi_{\omega_{t,f}}^T)u$. But the hypothesis $\det\left[\left<{\tilde{\xi}}_\omega^i,{\tilde{\xi}}_{\omega_{t,f}}^j\right>\right]\neq 0$ implies that $v\equiv0$ and then $ker\left((Id-\Pi_{\omega}^T)\circ(Id-\Pi_{\omega_{t,f}}^T)\right)=ker(Id-\Pi_{\omega_{t,f}}^T).$  

We then consider the map:

\begin{equation*}
\begin{array}{lrcl} \Psi : &\mathcal{U}& \longrightarrow &
\mathbb{R}\times\widetilde{C}^\infty_T(M,\mathbb{R}) \\
    & (t,f)& \longmapsto
    &\left(t,(Id-\Pi_{\omega}^T)\circ(Id-\Pi_{\omega_{t,f}}^T)
    (\mathring{s}^{\nabla_{t,f}})\right),
\end{array}
\end{equation*}
\medskip
where $\mathring{s}^{\nabla_{t,f}}$ is the zero integral part of the hermitian scalar curvature ${s}^{\nabla_{t,f}}$ of $(J_t,g_{t,f})$.

It follows from Proposition \ref{prop1} that $\Psi(t,f)=(t,0)$ if and only if
$(J_t,g_{t,f})$ is an extremal almost-K\"ahler metric. In
particular, $\Psi(0,0)=(0,0).$ 

Let ${{\alpha}}_{t,f}=J_td{f}-\delta^{g_t}\psi_{f}^t=J_td{f}-\delta^{g_t}{\mathbb{G}}_t\left((dJ_tdf)^{J_t,-}\right)=J_td{f}-\delta^{g_t}{\mathbb{G}}_t(D^{g_t}_{ {df}^{\sharp_{g_t}}}\omega),$ where ${\mathbb{G}}_t$ is the Green operator associated to the elliptic operator $P_t : \Omega^{J_t,-}(M) \rightarrow  \Omega^{J_t,-}(M).$
In order to extend the map $\Psi$ to Sobolev spaces, we give an explicit expression of $(Id-\Pi_{\omega_{t,f}}^T)  ({s}^{\nabla_{t,f}})$. A direct computation using (\ref{equa5}) shows that
\begin{equation}\label{equa8}
{s}^{\nabla_{t,f}}=\Delta^{g_{t,f}}F_{t,f}+g_{t,f}(\rho^{\nabla_t},\omega_{t,f}),
\end{equation}
where ${F_{t,f}}=\log\left(\frac{1}{2}\left(\left(1+g_t\left(d{{\alpha}}_{t,f},\omega\right)\right)^2+1-g_t(d{{\alpha}}_{t,f},d{{\alpha}}_{t,f})\right)\right)$ satisfying the relation $\omega_{t,f}^2=e^{F_{t,f}}\omega^2$. Then
\begin{equation} \label{equa15}
(Id-\Pi_{\omega_{t,f}}^T)  ({s}^{\nabla_{t,f}})=\Delta^{g_{t,f}}F_{t,f}+g_{t,f}(\rho^{\nabla_t},\omega_{t,f})-\sum_j\left<{s}^{\nabla_{t,f}},{\tilde{\xi}}_{\omega_{t,f}}^j\right>{\tilde{\xi}}_{\omega_{t,f}}^j.
\end{equation}
Let $\widetilde{W}_T^{p,k}$ be the completion of $\widetilde{C}^\infty_T(M,\mathbb{R})$ with respect to the Sobolev norm $\|\cdot\|_{p,k}$ involving derivatives up to order $k.$
We choose $p,k$ such that $pk>2n$ and the corresponding Sobolev space $\widetilde{W}_T^{p,k}\subset C^3_T(M,\mathbb{R})$ so that all coefficients are $C^0_T(M,\mathbb{R})$. Since $\widetilde{W}_T^{p,k}$ form an algebra relative to the standard multiplication of functions \cite{ada}, we dedudce from the expression (\ref{equa15}) that the extension of $\Psi$ to the Sobolev completion of
$\widetilde{C}^\infty_T(M,\mathbb{R})$ is a map $\Psi^{(p,k)}: \widetilde{\mathcal{U}}\subset\mathbb{R}\times\widetilde{W}_T^{p,k+4}\longrightarrow \mathbb{R}\times\widetilde{W}_T^{p,k}$.

Clearly $\Psi^{(p,k)}$ is a $C^1$ map (in a small enough open around $(0,0)$). Indeed, it is obtained by a composition of $C^1$ maps by Lemma \ref{lem5} and (\ref{equa15}).

As in \cite{leb-sim} and using Proposition \ref{prop1}, the differential of $\Psi^{(p,k)}$ at $(0,0)$ is given by
\begin{equation*}
\left(\textbf{T}_{(0,0)}\Psi^{(p,k)}\right)(t,f)=\left(t,t\delta^{g_{0}}\delta^{g_{0}}h-2\delta^{g_0}\delta^{g_0}(D^{g_0}d{f})^{J_0,-}\right),
\end{equation*}
where $h=\frac{d}{dt}|_{t=0}\,g_t$. 

The operator $L:=\frac{\partial\Psi}{\partial f}|_{(0,0)}$ given by $L({f})=-2\delta^{g_0}\delta^{g_0}(D^{g_0}d{f})^{J_0,-}$ is called
the {\it Lichnerowicz operator}. It is  a $4$-th order self-adjoint $T$-invariant elliptic linear
operator leaving invariant
$\left(\mathfrak{t}_{\omega}\right)^\bot$ since $L(f)=0$ for any $f\in\mathfrak{t}_{\omega}$. By a known result of the elliptic theory \cite{bes, wel}, we obtain the $L^2$-orthogonal splitting $\widetilde{C}^\infty_T(M,\mathbb{R})=ker(L)\oplus Im(L)$. Following the argument in \cite[Lemma 4]{apo-cal-gau-Fri 2}, any $f\in ker(L)$ gives rise to a Killing vector field in the
centralizer of $\mathfrak{t}=Lie(T).$ By the maximality of the torus $T$,
$f\in\mathfrak{t}_\omega.$ It follows that $L$ is an isomorphism of
$\widetilde{C}^\infty_T(M,\mathbb{R})$ and also from $\widetilde{W}_T^{p,k+4}$ to $\widetilde{W}_T^{p,k}$. Thus, $\textbf{T}_{(0,0)}\Psi^{(p,k)}$ is an isomorphism from $\mathbb{R}\oplus\widetilde{W}_T^{p,k+4}$ to $\mathbb{R}\oplus\widetilde{W}_T^{p,k}$. 
It follows from the {\it{inverse function theorem}} for Banach manifolds that $\Psi^{(p,k)}$ determines an isomorphism from an open neighbourhood $V$ of $(0,0)$ to an open neighbourhood of $(0,0)$. In particular, there exists $\mu>0$ such that for $|t|<\mu, \;\Psi^{(p,k)}\left(\Psi^{(p,k)}|_V^{-1}(t,0)\right)=(t,0)$. By Sobolev embedding, we can choose a $k$ large enough, such that $\widetilde{W}_T^{p,k+4}\subset\widetilde{C}^6_T(M,\mathbb{R})$. Thus, for $|t|<\mu$, $(J_t,g_{\Psi^{(p,k)}|_V^{-1}(t,0)})$ is an extremal almost-K\"ahler metric of regularity at least $C^4$ (so we ensure, in this case, that $grad_\omega s^{\nabla_{t,f}}$ is of regularity $C^1$).

By Proposition \ref{prop2}, the extremal vector field $Z_{\omega_{t,f}}^T=Z_\omega^T$ is smooth for any almost-K\"ahler metric $(J_t,g_{t,f})$. In particular, for an extremal almost-K\"ahler metric $(J_t,g_{t,f})$ of regularity $C^4$, the dual $ds^{\nabla_{t,f}}$ of $Z_\omega^T$ with respect to $\omega_{t,f}$ is of regularity $C^4$, then
the hermitian scalar curvature $s^{\nabla_{t,f}}$ of $(J_t,g_{t,f})$ is of regularity $C^5.$ From (\ref{equa8}), it follows that the hermitian scalar curvature is given by the pair of equations
\begin{eqnarray}
s^{\nabla_{t,f}}-g_{t,f}(\rho^{\nabla_t},\omega_{t,f})&=&\Delta^{g_{t,f}} (u),\label{e1}\\
e^{u}&=&\frac{\omega_{t,f}^2}{\omega^2}.\label{e2}
\end{eqnarray}

From (\ref{e1}), using the ellipticity \cite{bes} of the (Riemannian) Laplacian $\Delta^{g_{t,f}} $ and since the l.h.s of (\ref{e1})  is of H\"older class $C^{3,\beta}$ for any $\beta\in(0,1)$, it follows that $u$ is of class  $C^{5,\beta}.$
Following \cite{don2, wei}, the linearisation of the equation (\ref{e2}) $(\omega+d\alpha)\wedge d\dot{\alpha}=0$ together with the constraints $\delta^{g_t}\dot{\alpha}=0$ and $(d\dot{\alpha})^{J_t,-}=0$ form a linear elliptic system in $\dot{\alpha}$. Elliptic theory \cite{agm-dou-nir,bes} ensures that the almost-K\"ahler metric $g_{t,f}$ is of class $C^{5,\beta}$ as the volume form and we can prove that any extremal almost-K\"ahler metric of regularity $C^4$ is smooth by a bootstraping argument (in the K\"ahler case see \cite{leb-sim}).

We obtain then a smooth family of $T$-invariant extremal almost-K\"ahler structures $(J_t,\omega_t=\omega+d\alpha_t)$ defined for $|t|<\mu$. The main theorem follows from the Moser Lemma \cite{duf-sal}.

\subsection*{Acknowledgements} The author thanks V. Apostolov for his invaluable help and judicious advices and T. Dr\u{a}ghici for valuable suggestions. He is very grateful to P. Gauduchon, P. Guan and T. J. Li for useful discussion.



\begin{thebibliography}{99}

\bibitem {ada} R. A. Adams, {\it Sobolev spaces}. Pure and Applied Mathematics, Vol. 65. Academic Press, New York-London (1975).
\bibitem{agm-dou-nir} S. Agmon, A. Douglis {\&} L. Nirenberg, {\it Estimates near the boundary for solutions of elliptic partial differential equations satisfying general boundary conditions. II}. Comm. Pure Appl. Math. {\textbf{17}} (1964), 35--92.
\bibitem {apo-cal-gau-Fri 2} V. Apostolov, D. M. J. Calderbank, P. Gauduchon {\&} C. W. T{\o}nnesen-Friedman, {\it Extremal K\"{a}hler metrics on projective bundles over a curve}, preprint 2009, math.DG/0905.0498.
\bibitem {apo-dra} V. Apostolov {\&} T. Dr\u{a}ghici,  {\it The curvature and the integrability of almost-K\"ahler manifolds: a survey}, Fields Inst. Communications Series {\textbf{35}} AMS (2003), 25--53.
\bibitem{bar-pet-van} W. Barth, C. Peters {\&} A. Van de Ven, {\it Compact Complex Surfaces,} Ergeb. Math. Grenzgeb 3, Folge A. Series of Modern Surveys in Mathematics 4. Berlin: Springer (2004).
\bibitem {bes} A.L. Besse, {\it Einstein manifolds}, Ergeb. Math. Grenzgeb., Springer-Verlag, Berlin, Heidelberg, New York (1987).
\bibitem {cal} E. Calabi, {\it Extremal K\"ahler metrics}, in Seminar of Differential Geometry, S. T. Yau (eds), Annals of Mathematics Studies. {\textbf{102}} Princeton University Press (1982), 259--290.
\bibitem {cal2} E. Calabi, {\it Extremal K\"ahler metrics II}, in Differential Geometry and Complex Analysis, eds. I. Chavel and H. M. Farkas, Springer Verlag (1985), 95--114.
\bibitem {don} S. K. Donaldson, {\it Remarks on Gauge theory, complex geometry and 4-manifolds topology}, in `The Fields Medallists Lectures" (eds. M. Atiyah and D. Iagolnitzer), pp. 384--403, World Scientific, 1997.
\bibitem {don1} S. K. Donaldson, {\it Scalar curvature and stability of toric varieties}, J, Diff. Geom.  {\textbf{62}} (2002), 289--349.  
\bibitem {don2} S. K. Donaldson, {\it Two-forms on four-manifolds and elliptic equations,} Nankai Tracts Math., 11, World Sci. Publ., Hackensack, NJ, (2006), 153--172. 
\bibitem {dra-li-zha} T. Dr\u{a}ghici, T. J. Li {\&} W. Zhang, {\it Symplectic forms and cohomology decomposition of almost complex four-manifolds}, preprint 2008, math.DG/0812.3680.
\bibitem{fuj} A. Fujiki, {\it Moduli space of polarized algebraic manifolds and K\"{a}hler metrics}, [translation of Sugaku {\textbf{42}}, no 3 (1990), 231--243], Sugaku Expositions 5, no. 2 (1992), 173--191.
\bibitem{fuj-sch} A. Fujiki {\&} G. Schumacher, {\it The moduli space of extremal compact K\"{a}hler manifolds and generalized Weil-Petersson metrics}, Publ. Res. Inst. Math. Sci. {\textbf{26}} (1990), 101--183.
\bibitem{gau} P. Gauduchon, {\it Calabi's extremal K\"{a}hler metrics: An elementary introduction}. In preparation.
\bibitem{gau1} P. Gauduchon, {\it Hermitian connections and Dirac operators,} Boll. Un. Mat. Ital. B (7) 11 (1997), no. 2, suppl., 257--288.
\bibitem{kar} Y. Karshon, {\it Periodic Hamiltonian flows on four-dimensional manifolds,} Mem. Amer. Math. Soc. {\textbf{141}} (1999), no. 672.
\bibitem {kim-sun} J. Kim {\&} C. Sung, {\it Deformations of almost-K\"{a}hler metrics with constant scalar curvature on compact K\"{a}hler manifolds}, Ann. of Global Anal. and Geom. {\textbf{22}} (2002), 49--73.
\bibitem{kod} K. Kodaira, {\it Complex manifolds and deformation of complex structures}, Translated from the Japanese original by Kazuo Akao. Reprint of the 1986 English edition. Classics in Mathematics. Springer-Verlag, Berlin (2005).
\bibitem{kod-mor} K. Kodaira {\&} J. Morrow,  {\it Complex manifolds}, Holt, Rinehart and Winston, Inc. (1971).
\bibitem {leb-sim} C. LeBrun {\&} S.R. Simanca, {\it Extremal K\"{a}hler metrics and complex deformation theory}, Geom. Funct. Anal. {\textbf{4}} (1994), 179--200.
\bibitem {lej} M. Lejmi, {\it Extremal almost-K\"ahler metrics}, preprint 2009, to appear in International Journal of Mathematics, math.DG/0908.0859.
\bibitem {lej1} M. Lejmi, {\it Strictly nearly K\"ahler 6-manifolds are not compatible with symplectic forms}, \newline C. R. Math. Acad. Sci. Paris {\textbf{343}} (2006) 759--762.
\bibitem {lib} P. Libermann, {\it Sur les connexions hermitiennes,} C. R. Acad. Sci. Paris {\textbf{239}}, (1954). 1579--1581.

\bibitem {duf-sal} D. McDuff {\&} D. Salamon, {\it Introduction to Symplectic Topology} (second ed.), Oxford Mathematical Monographs, Oxford University Press (1998).
\bibitem {tos-wei} V. Tossati {\&} B. Weinkove, {\it The Calabi-Yau equation on the Kodaira-Thurston manifold}, preprint 2009, math.DG/0906.0634. 
\bibitem {tos-wei-yau} V. Tosatti, B. Weinkove {\&} S.T. Yau, {\it Taming symplectic forms and the Calabi-Yau equation}, Proc. Lond. Math. Soc.  {\textbf{3}} 97 no. 2 (2008), 401--424.
\bibitem {wei} B. Weinkove, {\it The Calabi-Yau equation on almost-K\"ahler four-manifolds}, J. Differential Geom. {\textbf{76}} (2007) 317--349.
\bibitem {wel} R.O. Wells, {\it Differential analysis in complex manifolds}, Second edition. Graduate Texts in Mathematics, 65, Springer-Verlag (1980).


\end{thebibliography}
\end{document}